\documentclass{amsart}

\usepackage{amscd}
\usepackage[shortlabels]{enumitem}
\usepackage[colorlinks=true, linkcolor=blue, citecolor=blue, urlcolor=blue]{hyperref}

\usepackage{amsthm}
\newtheorem{thm}[equation]{Theorem}
\newtheorem{lem}[equation]{Lemma}

\newtheorem{prop}[equation]{Proposition}
\newtheorem{tthm}[equation]{Trinomial Theorem}

\newtheorem{qprop}[equation]{Quartic Theorem}

\theoremstyle{definition}
\newtheorem{defn}[equation]{Definition}
\newtheorem{eg}[equation]{Example}

\theoremstyle{remark}

\newcommand{\R}{\mathbb{R}}

\newcommand{\C}{\mathbb{C}}
\newcommand{\Z}{\mathbb{Z}}

\DeclareMathOperator{\car}{char}

\newcommand{\B}{\mathcal{B}}

\newcommand{\abs}[1]{\left| #1 \right|}

\newcommand{\rt}{\alpha}
\newcommand{\rtt}{\beta}
\newcommand{\limit}{\lambda}
\newcommand{\limitt}{\mu}

\newcommand{\gb}{\bar{g}}
\newcommand{\fb}{\bar{f}}
\newcommand{\bc}{\bar{c}}
\newcommand{\qb}{\bar{q}}
\newcommand{\pb}{\bar{p}}

\renewcommand{\:}{\colon}

\newcommand{\stacks}[1]{\cite[\href{https://stacks.math.columbia.edu/tag/#1}{Tag #1}]{stacks-project}}

\begin{document}

\title{Root of the generic cubic as a power series in the discriminant}
\date{\texttt{\today}}


\begin{abstract}
An observation of J-P.~Serre implies that the generic monic cubic polynomial,
unique among generic monic polynomials of degree at least two, has a root that
is a power series in the discriminant; Serre asked for a formula.  We give formulas
that work over any field with an absolute value and in every characteristic.
Over a complete non-archimedean field we identify the root intrinsically: it
is the isolated root, the one farthest from the others.  We also answer the
next case of Serre's question, computing explicitly the distinguished ramified
quadratic factor of the generic monic quartic.  The methods combine Hensel's
lemma, identities for generalized binomial series, and elementary
non-archimedean analysis.
\end{abstract}

\author{Jason Bland}
\author{Skip Garibaldi}
\author{Joel Rosenberg}

\maketitle

\section{Introduction}

An observation of Jean-Pierre Serre implies that a generic cubic polynomial
has a root expressible as a power series in its discriminant, whereas generic
polynomials of higher degree do not.\footnote{Serre himself did not use this
imprecise language.  See Proposition~\ref{serre.prop}, \eqref{serre.cubic}, and
Definition~\ref{generic.case} for precise statements.}  He asked in a private communication for
a formula, which we provide over fields with an absolute value, including the
generic case, the real and complex numbers, and $p$-adic fields.

\begin{thm} \label{n3.soln}
Let $f(t) = t^3 + pt + q \in C[t]$ for $C$ a field of characteristic different
from $3$ with an absolute value, and suppose $p \ne 0$.  Write $\Delta$ for
the discriminant of $f$.  If the series
\begin{equation} \label{disc.series}
  \sum_{n \ge 0} \binom{3n}{n}\left( \frac{-\Delta}{27 p^3} \right)^{\!n}
\end{equation}
converges unconditionally to $\limit \in C$, then
\[
  f(\rt) = 0 \quad\text{for}\quad \rt := \frac{3q}{p}\limit \in C.
\]
\end{thm}

We call the series \eqref{disc.series} the \emph{discriminant series} and the
root $\rt$ the \emph{discriminant root}.  Convergence holds in the setting of Serre's question, answering it when $\car C \ne 3$.
The transformation $t\mapsto t-c_1/3$ gives a version of the theorem that applies to general cubics, see 
\eqref{n3.nondepressed}.  

In characteristic $3$, the denominator $27$ prevents using this formula, and
Example~\ref{3.eg} exhibits a depressed cubic with no root expressible as a
power series in its discriminant.  Theorem~\ref{3.soln} gives the appropriate analogue of Theorem \ref{n3.soln} for characteristic 3.

\subsection*{Non-archimedean characterization}
Our second main result, Theorem~\ref{dist.root.thm}, shows that over a
complete non-archimedean field the discriminant series converges exactly when
the cubic has an \emph{isolated} root, meaning the root farthest from the others, and
that the discriminant root is then the isolated root.  
The companion paper \cite{LongRoot} proves an analogous result for $\R$.
It uses real analytic tools, including hypergeometric functions; our methods
are algebraic and combinatorial.

\subsection*{Quartics}
Serre's observation about cubics is the case $n = 3$ of a statement about the
generic monic polynomial $f$ of every degree $n \ge 3$
(Proposition~\ref{serre.prop}): over the field $C$ of
Definition~\ref{generic.case}, $f$ factors uniquely as $r(t)u(t)$ with $r$ a
ramified irreducible quadratic and $u$ unramified.
This raises, for each $n$, the computational problem of writing $r(t)$ down
explicitly.  For $n = 3$ the factor $u$ is linear, so $r(t) = f(t)/(t - \rt)$
for the root $\rt$ given by Theorem~\ref{n3.soln} or \ref{3.soln}.  Our third
main result, Theorem~\ref{quartic.prop}, solves the problem for $n = 4$ when
$\car C \ne 2, 3$.  


\section{The discriminant}
Let $k$ be a field and $f = t^n + c_1 t^{n-1} + \cdots + c_n$ the
\emph{generic monic polynomial}, with $c_1, \ldots, c_n$ indeterminates.
Embedding $k[c_1, \ldots, c_n]$ into $k[x_1, \ldots, x_n]$ via $c_i \mapsto
(-1)^i e_i(x_1, \dots, x_n)$, where $e_i$ is the degree $i$ elementary
symmetric polynomial, identifies $f$ with $\prod_j (t - x_j)$.
By the
Fundamental Theorem of Symmetric Polynomials, $k[c_1, \ldots, c_n]$ is the
subring fixed by the symmetric group $S_n$ permuting the $x_j$.

Define $\delta := \prod_{1 \le i < j \le n} (x_i - x_j) \in k[x_1, \ldots,
x_n]$ and $\Delta := \delta^2$.  Because $\Delta$ is fixed under permutations of the
$x_j$, it belongs to $k[c_1, \ldots, c_n]$.  For example, in the case $n
= 3$ we have:
\begin{equation} \label{disc.eq}
\Delta = c_1^2 c_2^2 - 4c_2^3 - 4c_1^3 c_3 + 18 c_1 c_2 c_3 - 27 c_3^2 \quad \in k[c_1, c_2, c_3].
\end{equation}
This $\Delta$ is the discriminant of the generic monic polynomial $f$.  One
can use it to define the discriminant for any field $K \supseteq k$ and any monic polynomial $g \in K[t]$ of degree $n$: it is the image of $\Delta$ in $K$
under the homomorphism that sends $c_i$ to the coefficient of $t^{n-i}$ in $g$.

In case the characteristic of $k$ is 2, $\delta$ is preserved by permuting the
$x_j$, so $\delta$ itself belongs to $k[c_1, \ldots, c_n]$.  For example, when
$n = 3$ and $\car k = 2$, we have:
\begin{equation} \label{2.disc}
\Delta = \delta^2 \quad \text{for} \quad \delta = c_1 c_2 + c_3.
\end{equation}
In summary, for $n \ge 2$, the element
\begin{equation} \label{pi.def}
\pi := \begin{cases} \Delta & \text{if $\car k \ne 2$;} \\
\delta & \text{if $\car k = 2$}
\end{cases}
\end{equation}
belongs to $k[c_1, \ldots, c_n]$.

The following is standard.  We provide a proof for lack of an ideal reference.

\begin{lem} \label{pi.prime}
    $\pi$ is an irreducible element of $k[c_1, \ldots, c_n]$.
\end{lem}

\begin{proof}
Suppose $\pi = gh$ with $g, h \in k[c_1, \dots, c_n]$, and write
$\pi = \delta^\varepsilon$, where $\varepsilon = 1$ in characteristic $2$ and
$2$ otherwise.  In the unique factorization domain $k[x_1, \dots, x_n]$, the
linear factors $x_i - x_j$ are pairwise non-associate primes, and $g$ divides
$\delta^\varepsilon$, so $g = u \prod_{i<j} (x_i - x_j)^{a_{ij}}$ with $u \in
k^\times$ and $0 \le a_{ij} \le \varepsilon$.  Each $\sigma \in S_n$ fixes $g$
and permutes the $x_i - x_j$ up to sign; as $S_n$ acts transitively on pairs,
all $a_{ij}$ equal some $a$, so $g = u\delta^a$.  When $\varepsilon = 2$, the
case $a = 1$ is impossible: a transposition fixes $g = u\delta$ but negates
$\delta$, so $2u\delta = 0$, contradicting $\car k \ne 2$.  Hence $a = 0$ or
$\varepsilon$, so $g$ or $h$ is a unit, proving irreducibility.
\end{proof}

Because $\pi$ is prime, localizing $k[c_1, \ldots, c_n]$ at $(\pi)$ and
completing gives a complete discrete valuation ring $A$ with uniformizer $\pi$
\stacks{05GH}.  Write $C$ for its fraction field and $\kappa := A/(\pi)$ for its
residue field.  In \autoref{abs.sec} we view $A$ as the closed unit disk of an
absolute value on $C$.

\section{A question for cubics and every higher degree} \label{serre.obs}

Serre's observation about cubic polynomials is a corollary of a more general
statement, that ``generic'' monic polynomials of degree $n \ge 2$ have a
distinguished quadratic factor.  We continue the notation of the previous section.

We use the notion of ramification from \cite[p.~37, 7.1]{SeciV}.  Let $v$ be the
valuation on $C$ with $v(\pi) = 1$, and let $L$ be a field that is
finite-dimensional over $C$.  Because $A$ is complete, $v$ extends uniquely to
a discrete valuation $v_L \colon L^\times \to \mathbb{Q}$ \cite[\S5.1, A]{Rib}.  Its valuation ring
is the integral closure of $A$ in $L$, its residue field is a finite extension
of $\kappa$, and its value group $v_L(L^\times)$ contains $v(C^\times) = \Z$
with finite index $e$, the \emph{ramification index}.  The field $L$ is
\emph{unramified} over $C$ if $e = 1$ and the residue field is separable over
$\kappa$, and \emph{ramified} otherwise.  A monic polynomial $h \in A[t]$ that is
irreducible over $C$ is called ramified or unramified according as the field
$C[t]/(h)$ is.

For the following result, see p.~56 of \cite{GMS} or the slightly different
presentation of the same material on p.~85 of \cite{SeciV}.

\begin{prop} \label{serre.prop}
There exist unique monic polynomials $r(t), u(t) \in A[t]$ such that $f(t) =
r(t) u(t)$; $\deg r(t) = 2$; $r$ is irreducible in $C[t]$ and ramified; and $u$
is unramified and irreducible (if $n > 2$) or $1$ (if $n = 2$) in $C[t]$.
\end{prop}

The intuition behind the proposition is that a generic polynomial has distinct
roots, but if we impose the extra condition that $\Delta = 0$ (which is
effectively what happens when we pass to the reduction $\bar f \in \kappa[t]$)
then generically it has exactly two roots that are equal and the rest pairwise
distinct.  The polynomial $r$ encodes the lifts of the two roots of $\bar f$
that are equal.

The proposition raises a general computational question:
\begin{equation} \label{serre.q}
\text{\emph{For each $n \ge 2$, what is the formula for the ramified quadratic factor $r(t)$?}}
\end{equation}
For $n = 2$ the answer is $r = f$; when $\car k \ne 2$, its irreducibility
amounts to the discriminant, a uniformizer of $A$, not being a square in $C$.  For
$n = 3$, Theorem~\ref{n3.soln} or \ref{3.soln} provides a root $\rt$ of $f$, and
$r(t) = f(t)/(t - \rt)$, since $t - \rt$ and the irreducible quadratic $r$ both
divide $f$.  The case $n = 4$ is the subject of \autoref{quartic.sec}.

\section{Why the generic monic cubic has a distinguished root} \label{why.n3.sec}

We give a self-contained proof, via Hensel's lemma, of the following special case of 
Proposition~\ref{serre.prop}:
\begin{equation} \label{serre.cubic}
\parbox{0.85\textwidth}{\emph{The polynomial $t^3 + c_1 t^2 + c_2 t + c_3 \in C[t]$ has exactly one root in $C$.}}
\end{equation}
Every $C$-root of $f(t) := t^3 + c_1 t^2 + c_2 t + c_3$ lies in $A$, since
$f$ is monic and $A$ is integrally closed.  If $\car k \ne 2$, two $C$-roots
would force the third into $C$ by their sum, making $\Delta = \delta^2$ a square
in $C$.  As $\Delta$ is a uniformizer, this is impossible; thus $f$ has at most
one $C$-root.

Write a bar for reduction modulo $\pi$.  We first treat $\car k = 3$, working
directly with $f$, and then $\car k \ne 3$ by depressing the cubic.

\subsection*{Proof \texorpdfstring{of \eqref{serre.cubic}}{} when \texorpdfstring{$\car k = 3$}{char k = 3}}
In characteristic $3$, \eqref{disc.eq} reduces to
\begin{equation} \label{3.disc}
\Delta = c_1^2 c_2^2  - c_1^3 c_3 - c_2^3,
\end{equation}
and $\pi = \Delta$.  The image $\fb \in \kappa[t]$ of $f$ is a cubic polynomial
whose discriminant is zero.  Setting $u := \bc_2/\bc_1$ (note $\bc_1 \ne 0$, as
$\Delta$ does not divide $c_1$), we find
\[
-\bc_1^3 \fb(u) = \bar{\Delta} = 0 \quad \text{and} \quad \fb'(u) = 3\bc_2 = 0
\]
in $\kappa$, so $u$ is a repeated root of $\fb$.  Since the three roots of
$\fb$ sum to $-\bc_1$, we find that $\bc_2/\bc_1 - \bc_1$ is a root distinct
from $u$.  In summary,
\begin{equation} \label{3.factor}
\fb(t) = (t - (\bc_2/\bc_1 - \bc_1))(t - \bc_2/\bc_1)^2 \quad \in \kappa[t].
\end{equation}
By Hensel's Lemma (\stacks{04GM} or \cite[\S{III.4.3}]{Bou:ac}), the simple
root of $\fb$ lifts to a root $\rt \in A$ of $f$ with $\rt \equiv c_2/c_1 - c_1
\bmod{\pi}$, proving \eqref{serre.cubic}.

\subsection*{Depressing the cubic}
Now suppose $\car k \ne 3$.  We may substitute $t \mapsto t - c_1/3$ to obtain
\begin{equation} \label{depressed}
g(t) := f(t - c_1/3) = t^3 + pt + q
\end{equation}
for
\begin{equation} \label{pq.def}
p = c_2 - c_1^2/3 \quad \text{and}  \quad q = 2c_1^3/27 - c_1 c_2/3 + c_3.
\end{equation}
The substitution translates roots and preserves the discriminant, so
\begin{equation} \label{n23.disc}
\Delta = -4p^3 - 27q^2.
\end{equation}
It also identifies $k[c_1, c_2, c_3]$ with the polynomial ring $k[c_1, p,
q]$, in which $\pi$ is the irreducible polynomial $-4p^3 - 27q^2$ if $\car k
\ne 2$, and is $q$ if $\car k = 2$ (compare \eqref{2.disc}).  Thus $\pi$ does
not divide $p$, so $p$ is a unit of $A$; likewise $q$ is a unit unless $\car k
= 2$.

\subsection*{Proof \texorpdfstring{of \eqref{serre.cubic}}{} when \texorpdfstring{$\car k \ne 2, 3$}{char not 2, 3}}
The relation $-4\pb^3 - 27\qb^2 = 0$ gives
\[
\gb(t) = (t - 3\qb/\pb) (t + 3\qb/2\pb)^2.
\]
The root $3\qb/\pb$ is simple because $\qb \ne 0$ and $\car k \ne 2, 3$.
Hensel's Lemma lifts it to a root
$\rt \in A$ of $g$ with $\rt \equiv 3q/p \bmod{\pi}$, proving
\eqref{serre.cubic}.

\subsection*{Proof \texorpdfstring{of \eqref{serre.cubic}}{} when \texorpdfstring{$\car k = 2$}{char k = 2}}
As $\pi = q = c_1 c_2 + c_3$, $g(0) = q \in (\pi)$ and $g'(0) = p$ is a unit, so Hensel's Lemma gives
a root $\rt \in A$ with $\rt \equiv 0 \equiv q/p \bmod{\pi}$.  The quotient
$Q(t) = g(t)/(t-\alpha) = t^2 + \rt t + (\rt^2 + p) \in A[t]$ has reduction
$\bar Q = t^2 + \pb$.   The residue field is
$\kappa = k(\bc_1, \bc_2)$ with $\bc_3 = \bc_1 \bc_2$, and in it $\pb = \bc_2 +
\bc_1^2$ is not a square, since $\partial \pb/\partial \bc_2 = 1$.  Thus $Q$ has no root in $A$ or $C$,
proving uniqueness.

\subsection*{An implicit answer to Serre's question}
Serre asked for an explicit formula for the root in \eqref{serre.cubic}.
There is a quick answer of a sort.  The proofs above give the residue modulo $\pi$ of the root $\rt \in A$ of
$f$ (after undoing the translation $t \mapsto t - c_1/3$ where it was used), a
simple root of $\bar f$.  Newton's method, cf.\ \cite[II.2.2]{Se:Cours},
then computes $\rt$ to any desired precision: starting from any lift $\rt_0 \in
A$ of this residue, the iterates $\rt_{i+1} := \rt_i - f(\rt_i)/f'(\rt_i)$
satisfy $\rt_i \equiv \rt \pmod{\pi^{2^i}}$.  This amounts to an implicit formula for $\rt$;
Theorems~\ref{n3.soln} and \ref{3.soln} give explicit ones.

\section{Absolute values} \label{abs.sec}

We are going to answer Serre's question \eqref{serre.q} for cubics in a setting that includes the generic case as an example.
Let $C$ be a field equipped with an \emph{absolute value}
$\abs{\cdot}\: C \to \R_{\ge 0}$, that is, a map with $\abs x = 0$ iff $x = 0$,
$\abs{xy} = \abs x\,\abs y$, and $\abs{x + y} \le \abs x + \abs y$.  Such a $C$
is a metric space under $\abs{x - y}$, and is \emph{complete} if every Cauchy
sequence converges.  The completion of $C$ with respect to this metric is also a field \cite[\S1.5]{Rib}.
The absolute value is \emph{non-archimedean} if $\abs{n \cdot 1_C} \le 1$ for
all $n \in \Z$, in which case it satisfies the ultrametric inequality
$\abs{x + y} \le \max\{\abs x, \abs y\}$ \cite[\S1.2, E]{Rib}; then the closed
unit disk $A = \{\abs x \le 1\}$ is a local ring with maximal ideal the open unit
disk $\{\abs x < 1\}$.

\begin{eg} \label{pi.adic}
Given a prime element $\pi$ in a unique factorization domain $D$, the $\pi$-adic
valuation $v_\pi(d)$ (the exponent of $\pi$ in $d$, with $v_\pi(0) = \infty$)
yields a non-archimedean absolute value $\abs{x} := 2^{-v_\pi(x)}$ on the
fraction field $K$ of $D$.  Completing $K$ gives a field $C$ whose closed unit
disk $A$ is a complete local ring with maximal ideal $(\pi)$ \stacks{05GH}.
\end{eg}

\begin{defn} \label{generic.case}
The \emph{generic case} is Example~\ref{pi.adic} with $D := k[c_1, \ldots,
c_n]$, $k$ a field, and $\pi$ as in \eqref{pi.def} (the discriminant when $\car
k \ne 2$); here $n = 3$ for cubics and $n = 4$ for quartics.  The resulting $A$
and $C$ are those of \autoref{serre.obs}.  Each $a \in A$ can be written as $a =
\sum_{n \ge 0} a_n \pi^n$ with each $a_n \in K$ either $0$ or of $\pi$-adic
valuation $0$, which is why we speak of a ``power series in the
discriminant.''  The phrase is loose because the coefficients $a_n$ are not
determined by $a$; for example, $1 = (1 + \pi) - \pi$.  The phrase may also be misleading when $\car k = 2$, for then $\pi = \delta$ not $\Delta$; nonetheless, Theorem \ref{n3.soln} still provides a formula that is sensibly described as a power series in the discriminant.
\end{defn}

\begin{eg}[characteristic 3] \label{3.eg}
Some depressed cubics have no root expressible as a power series in the
discriminant.  Apply Example~\ref{pi.adic} with $D := k[p, q]$, $k$ of
characteristic $3$, and $\pi := p$.  We claim
\begin{equation} \label{3.eg.1}
\text{\emph{$f(t) := t^3 + pt + q \in C[t]$ has no root in $C$.}}
\end{equation}
If it had a root $\rt \in C$, then $\rt \in A$, since $f$ is monic and the
discrete valuation ring $A$ is integrally closed; reducing modulo $\pi = p$ --- under the
identification $A/(\pi) \cong k(q)$, where $\bar f = t^3 + q$ --- would give
$-q$ a cube root in $k(q)$, which is false.  Since $\Delta = -p^3$ by
\eqref{3.disc}, a power series in $\Delta$ here is a power series in $p$, hence
an element of $C$; so this $f$ is the desired example.
\end{eg}

\subsection*{Series}
Convergence of sequences and series in $C$ is defined as over $\R$.  In case $C$ is 
non-archimedean and complete, $\sum_n c_n$ converges iff $\abs{c_n} \to 0$
\cite[\S1.1.8, Prop.~1]{BGR}, and every convergent series converges
\emph{unconditionally} --- that is, every rearrangement converges to the same
sum \cite[\S1.1.8, Cor.~4]{BGR}.  Over $\R$ or $\C$ with the usual absolute value, unconditional convergence
is the same as absolute convergence, and the radius of convergence of a power
series is given by the root test as usual.  

The following lets us prove an
identity in a formal power-series ring and then evaluate it in $C$.  We write
$a(c)$ or $a(z)\vert_{z=c}$ for the result of substituting $c \in C$ into $a(z) = \sum_n a_n z^n \in
\Z[[z]]$.

\begin{prop} \label{formal}
Let $h \in \Z[x_1, \ldots, x_s]$ and $a_1, \ldots, a_s \in \Z[[z]]$.  If $c \in
C$ is such that every $a_i(c)$ converges unconditionally in $C$, then the 
series obtained by plugging $c$ into $h(a_1(z), \ldots, a_s(z))$ converges unconditionally and its sum is $h(a_1(c), \ldots, a_s(c))$.
\end{prop}

\begin{proof}
By induction on the construction of $h$ it suffices to treat integer scaling,
addition, and a product of two terms; only the product needs argument, namely
the Cauchy product theorem for unconditionally convergent series, and we may
assume $C$ complete.  If the absolute value is archimedean, then by Ostrowski's
Theorem \cite[\S1.6, S]{Rib} we may take $C = \R$ or $\C$ with the usual absolute value, where unconditional
convergence is absolute convergence and this is classical.  If it is
non-archimedean, convergence means terms tending to $0$, and the ultrametric
inequality shows the tails of the Cauchy product tend to $0$.
\end{proof}

\section{Generalized binomial series} \label{gbs.sec}

We collect in this section material on a family of power series that underlies the series appearing in 
 Theorems~\ref{n3.soln}, \ref{trinomial.thm}, and
 Theorem~\ref{3.soln}.  
 
For integers $n, r$ and a non-negative integer $m$, the \emph{Fuss--Catalan
number} is
\[
  A_0(n,r) := 1, \qquad
  A_m(n,r) := \frac{r}{m} \binom{mn+r-1}{m-1} \quad \in \Z \quad (m \ge 1).
\]
When $mn + r \ne 0$, this is the standard expression $\frac{r}{mn+r}
\binom{mn+r}{m}$; the form above also covers the case $mn + r = 0$.   For example, $A_m(0,r)$ is the binomial coefficient
$\binom{r}{m}$, the numbers $A_m(2,1)$ are the Catalan numbers, and the numbers
$A_m(3,1)$ --- which we use most --- form Sequence A001764 in
OEIS~\cite{OEIS}.

The
associated power series is
\begin{equation} \label{B.def}
  \B_{n,r}(z) := \sum_{m \ge 0} A_m(n,r) z^m \quad \in \Z[[z]],
\end{equation}
with constant term $\B_{n,r}(0) = A_0(n,r) = 1$.  The special case $\B_n(z) :=
\B_{n,1}(z)$ is the \emph{generalized binomial series} of
\cite[\S5.4]{Concrete}.  

We will use three formal identities.  The first is the defining functional
equation of $\B_n$,
\begin{equation} \label{B.functional}
  z\,\B_n(z)^n = \B_n(z) - 1,
\end{equation}
recorded as \cite[p.~200, (5.59)]{Concrete}; equivalently $\B_n(z) = 1 + z\,
\B_n(z)^n$.  This equation determines $\B_n$ among power series with constant
term $1$: comparing coefficients of $z^m$ on both sides expresses the
coefficient of $z^m$ in $\B_n$ in terms of coefficients of lower degree.  The second is 
\begin{equation} \label{B.central}
  \sum_{m \ge 0} \binom{nm}{m} z^m = \frac{\B_n(z)}{\,n - (n-1)\B_n(z)\,},
\end{equation}
which is \cite[(5.61)]{Concrete}
with the exponent parameter set to $0$.
For $n = 3$ it reads $\sum_m
\binom{3m}{m} z^m = \B_3/(3 - 2\B_3)$.  The third says that the series
$\B_{n,r}$ are the powers of $\B_n$ \cite[(5.60)]{Concrete}:
\begin{equation} \label{B.power}
  \B_{n,r}(z) = \B_n(z)^r.
\end{equation}

\subsection*{The trinomial root}
Generalized binomial series arose from the solution of trinomial equations: a
root of $t^n + pt + q = 0$ was written as a power series in $q^{n-1}/p^n$
already in the 1700s; see \cite[\S\S38--40]{Lambert},
\cite[No.~12]{Lagrange:res}; more modern sources include \cite{Sturmfels}, or
\cite[Th.~10]{WildRubine}.  In our notation:

\begin{tthm} \label{trinomial.thm}
Let $C$ be a field with an absolute value and suppose $p, q \in C$ with $p \ne
0$.  If the series
\begin{equation} \label{trinomial.series}
 \B_n \left( \frac{q^{n-1}}{(-p)^n} \right)
\end{equation}
for some $n \ge 2$ converges unconditionally to some $\limitt \in C$, then $\rtt^n + p\rtt + q = 0$
in $C$ for $\rtt = -q\limitt/p$.
\end{tthm}

One way to see this is to substitute the functional equation
\eqref{B.functional}, in the form $z\,\B_n(z)^n = \B_n(z) - 1$, into $t^n + pt
+ q$ via Proposition~\ref{formal}.  The proofs of Theorems~\ref{n3.soln}
and~\ref{3.soln} below follow the same pattern, with other identities in
$\Z[[z]]$ in place of \eqref{B.functional}.  We call \eqref{trinomial.series}
the \emph{trinomial series} and, when it converges unconditionally, we call
$\rtt$ the \emph{trinomial root}.

\subsection*{Convergence}
The identities \eqref{B.functional}--\eqref{B.power} also determine where the
series converge over a non-archimedean field.

\begin{prop} \label{B.conv}
Let $C$ be complete with a non-archimedean absolute value, and let $n \ge 2$ and
$r \ne 0$ be integers.  Then $\B_{n,r}(a)$ converges for $a \in C$ if and only if $\abs{a} <
1$, and the same holds for $\sum_m \binom{nm}{m} a^m$ provided $n \ne 0$ in $C$.
More precisely, $\abs{A_m(n,r)} = 1$ for infinitely many $m$, and
$\abs{\binom{nm}{m}} = \abs{n}$ for infinitely many $m$.
\end{prop}

\begin{proof}
The coefficients are integers, so the series converge when $\abs{a} < 1$, and
the last sentence of the proposition gives divergence when $\abs{a} \ge 1$.  As
$\binom{nm}{m} = n \binom{nm-1}{m-1}$ for $m \ge 1$, it suffices to show that
$\B_n^r = \B_{n,r}$ (see \eqref{B.power}) and $T(z) := \sum_{m \ge 1}
\binom{nm-1}{m-1} z^m$ each have infinitely many coefficients of absolute value
$1$.  Equivalently, writing $\kappa$ for the residue field of $C$, the
reductions of $\B_n^r$ and $T$ to $\kappa[[z]]$ are not polynomials, which is
what we show.

Let $b \in \kappa[[z]]$ be the reduction of $\B_n$.  By \eqref{B.functional},
$b = 1 + z + \cdots$ and $z = (b-1)/b^n$.
We claim that substituting $t \mapsto b$ defines an injection $\kappa[t]
\hookrightarrow \kappa((z))$.  Indeed, any nonzero element of $\kappa[t]$ can be
written as $(t-1)^j Q(t)$ with $Q(1) \ne 0$.  Since $b - 1 = z + \cdots$, such
an element maps to $(b-1)^j Q(b) = Q(1) z^j + \cdots \ne 0$.  As $\kappa((z))$
is a field, the injection extends to an embedding $\kappa(t) \hookrightarrow
\kappa((z))$.  

If the reduction $b^r$ of $\B_n^r$ were a polynomial $P \in \kappa[z]$, then
$P(0) = 1$ and $t^r = P\bigl((t-1)/t^n\bigr)$ in $\kappa(t)$.  As $n \ge 2$,
the right side takes the value $P(0) = 1$ at $t = \infty$, whereas the left
side has a zero or a pole there.

For $T$, \eqref{B.central} gives $1 + nT = \B_n/(n - (n-1)\B_n)$, where the denominator has constant term $1$ and so is a
unit in $\Z[[z]]$.  Subtracting $1$ gives $nT = n(\B_n - 1)/(n - (n-1)\B_n)$,
and cancelling $n$ in $\Z[[z]]$ gives $T = (\B_n - 1)/(n - (n-1)\B_n)$.  If the
reduction of $T$ were a polynomial $P \in \kappa[z]$, then $P(0) = 0$ and
$(t-1)/(n - (n-1)t) = P\bigl((t-1)/t^n\bigr)$ in $\kappa(t)$.  At $t = \infty$ the right side
is $0$, whereas the left side is $-1/(n-1)$ if $n \ne 1$ in $\kappa$ and has a
pole otherwise.
\end{proof}

For example, if $C$ is complete and non-archimedean, then for $p \ne 0$ and
every $n \ge 2$ the trinomial series \eqref{trinomial.series} converges if and only if
$\abs{q}^{n-1} < \abs{p}^n$.

In the excluded case where $n = 0$ in $C$, each $\binom{nm}{m}$ with $m \ge 1$ is a multiple of $n$, so $\sum_m \binom{nm}{m} a^m = 1$ for every $a \in C$.

\section{The discriminant root (\texorpdfstring{$\car \ne 3$}{char not 3})} \label{n3.sec}

For a non-depressed cubic $f(t) = t^3 + c_1t^2 + c_2 t + c_3$, take $p, q$ as
in \eqref{pq.def} with $p \ne 0$.  If the discriminant series
\eqref{disc.series} converges unconditionally to $\limit$, then
\begin{equation} \label{n3.nondepressed}
-\frac{c_1}{3} + \frac{3q}{p} \limit
\end{equation}
is a root of $f$ by Theorem~\ref{n3.soln}.  That is, the theorem provides a root also for non-depressed cubics.

In the generic case from Definition \ref{generic.case}, $\abs{\Delta} < 1 = \abs{3p}^3$, so the series converges
unconditionally and Theorem \ref{n3.soln} answers Serre's question in characteristic $\ne 3$.

\begin{proof}[Proof of Theorem \ref{n3.soln}]
The discriminant series is the evaluation at $z = -\Delta/(27p^3)$ of
\[
  s(z) := \sum_{n \ge 0} \binom{3n}{n} z^n \quad \in \Z[[z]].
\]
We claim that $s$ satisfies the identity
\begin{equation} \label{binom3.id}
  (4 - 27z)\, s(z)^3 = 1 + 3 s(z)
\end{equation}
in $\Z[[z]]$.  By \eqref{B.central} with $n = 3$ we have $s = \B_3/(3 - 2\B_3)$,
and the functional equation \eqref{B.functional} for $n = 3$ gives $z =
(\B_3 - 1)/\B_3^3$.  Substituting these two expressions, both sides of
\eqref{binom3.id} become the same rational function of $\B_3$; clearing the
common denominator, their difference is the zero polynomial in $\B_3$, so
\eqref{binom3.id} holds.\footnote{The companion paper \cite[Lemma~11]{LongRoot}
proves \eqref{binom3.id} by a complementary route, differentiating
\eqref{B.functional} to express $s$ through $\B_3$ and its derivative.  Another proof is \cite[p.~192, (6.21)]{Stanley2}.  The identity is also recorded in
OEIS~\cite{OEIS} as a relation between Sequence A378483, the coefficients of
$s(z)^3$, and A005809, the coefficients of $s(z)$.}

Therefore, for $d := -\Delta/(27p^3)$, by Proposition \ref{formal} and the
unconditional convergence to $\limit$, we have
\[
(4 - 27d) \limit^3 = 1 + 3\limit
\]
in $C$.  Plugging in to $f$, we find:
\[
f(3q \limit/p) = q \left( - (4 - 27d) \limit^3 + 3 \limit + 1\right) = 0. \qedhere
\]
\end{proof}

\section{Characteristic 3} \label{deg3no.sec}

We answer Serre's question \eqref{serre.q} for cubics in characteristic $3$.  By
Example~\ref{3.eg}, the formula must apply to cubics with $c_1 \ne 0$.  It uses
the series $\B_{3,-1}$ of Section~\ref{gbs.sec}.

\begin{thm} \label{3.soln}
Let $f(t) = t^3 + c_1 t^2 + c_2 t + c_3 \in C[t]$ for $C$ a field of
characteristic 3 with an absolute value, and suppose $c_1 \ne 0$.  If the
series
\[
\B_{3,-1}(\Delta/c_1^6)
\]
converges to $\limit \in C$, then
\[
f(\rt) = 0 \quad \text{for} \quad \rt = \frac{c_2}{c_1} - c_1\limit \quad \in C.
\]
\end{thm}

In the generic case, $\abs{\Delta} < 1 = \abs{c_1}^6$, so the theorem answers
Serre's question.  Every absolute value in characteristic $3$ is non-archimedean,
so convergence is unconditional.  When $\Delta = 0$, the series is $1$ and
$\rt = c_2/c_1 - c_1$, as in \eqref{3.factor}.

\begin{proof}[Proof of Theorem \ref{3.soln}]
Put $L(z) := \B_{3,-1}(z)$.\footnote{The coefficients of $(1-L(z))/z$ are
Sequence A006013 in \cite{OEIS}.}  By \eqref{B.power} and \eqref{B.functional},
\begin{equation} \label{3.soln.1}
L(z)^2(1-L(z)) = \frac{\B_3(z)-1}{\B_3(z)^3} = z.
\end{equation}
Proposition~\ref{formal} gives $\limit^2(1-\limit) = \Delta/c_1^6$.  In
characteristic $3$ one checks that $c_1^3 f(t) = (c_1 t - c_2)^2(c_1 t - c_2 +
c_1^2) - \Delta$, and $c_1\rt - c_2 = -c_1^2\limit$, so
\[
c_1^3 f(\rt) = c_1^6\limit^2(1-\limit) - \Delta = 0. \qedhere
\]
\end{proof}

The roots of Theorems~\ref{n3.soln} and~\ref{3.soln} are two faces of one
formula: over $\Z$, the identity $\binom{3n}{n} = 3\binom{3n-1}{n-1}$ cancels
the factors of $3$ in the denominators of \eqref{n3.nondepressed}, and the
resulting formula reduces modulo $3$ to the root of Theorem~\ref{3.soln}.  This
gives another proof of Theorem~\ref{3.soln}, longer than the one above.

\section{Isolated roots over complete non-archimedean fields} \label{nonarch.sec}

Throughout this section and the next, $C$ is complete with a non-archimedean
absolute value.  The absolute value extends uniquely to an algebraic closure, and in particular defines the lengths of the roots $x_1, x_2, x_3$ of a cubic $f(t) = t^3 + c_1 t^2 + c_2 t + c_3 \in C[t]$ and the differences $x_1 - x_2$, $x_2 - x_3$, $x_3 - x_1$.
The differences sum to zero, so the ultrametric inequality implies that 
all distances between roots are equal or there is one
short side of the ``triangle'' formed by the roots.  In the latter case, we call the root opposite the short side the
\emph{isolated} root.  Our second main result determines when it exists and
identifies it as the discriminant root, see Theorem~\ref{dist.root.thm}.

We start with a preparatory lemma.

\begin{lem}\label{length.dichotomy}
Let $y_1, y_2, y_3$ be elements of an algebraic closure of $C$ with $y_1 + y_2
+ y_3 = 0$, and set $P := y_1 y_2 + y_1 y_3 + y_2 y_3$ and $Q := -y_1 y_2
y_3$.  There is a unique shortest $y_i$ if and only if $\abs{Q}^2 <
\abs{P}^3$.  In that case its length is $\abs{Q}/\abs{P}$ and the other two
have length $\abs{P}^{1/2}$; otherwise all three have length $\abs{Q}^{1/3}$.
\end{lem}

\begin{proof}
Let $L := \max_i \abs{y_i}$.  As no $y_i$ is strictly longer than the other
two, either some $y_i$ is strictly shorter than the other two or all three have
length $L$.  In the first case, say $\abs{y_1} < \abs{y_2} = \abs{y_3} = L$,
the relation $y_2 + y_3 = -y_1$ gives $P = y_2 y_3 + y_1(y_2 + y_3) = y_2 y_3 -
y_1^2$, whence $\abs{P} = L^2$ and $\abs{Q} = \abs{y_1} L^2$; thus
$\abs{Q}/\abs{P} = \abs{y_1} < L = \abs{P}^{1/2}$ and $\abs{Q}^2 = \abs{y_1}^2
\abs{P}^2 < \abs{P}^3$.  In the second case, $\abs{Q} = L^3$ and $\abs{P} \le
L^2$, so $\abs{Q}^2 \ge \abs{P}^3$.
\end{proof}

For the cubic $f$, put $H := c_1^2 - 3c_2$.  When $H \ne 0$, the series in
Theorems~\ref{n3.soln} and~\ref{3.soln} are both evaluated at $\Delta/H^3$: if
$\car C \ne 3$, then $H = -3p$ for $p$
as in \eqref{pq.def}, so $\Delta/H^3 = -\Delta/(27p^3)$; and if $\car C = 3$,
then $H = c_1^2$, so $\Delta/H^3 = \Delta/c_1^6$.  Accordingly, we say that the
\emph{discriminant root} of $f$ is \emph{defined} if $H \ne 0$ and the
discriminant series \eqref{disc.series} (if $\car C \ne 3$) or
$\B_{3,-1}(\Delta/c_1^6)$ (if $\car C = 3$) converges; it is then the root
\eqref{n3.nondepressed}, respectively the root of Theorem~\ref{3.soln}.

\begin{thm}\label{dist.root.thm}
Let $f(t) = t^3 + c_1 t^2 + c_2 t + c_3 \in C[t]$ have discriminant $\Delta$,
and put $H := c_1^2 - 3c_2$.  Then $f$ has an isolated root if and only if
$\abs{\Delta} < \abs{H}^3$.  This is also exactly the condition under which the
discriminant root of $f$ is defined, and that root is then the isolated root.
Moreover:
\begin{enumerate}[(a)]
\item \label{dist.isolated} If $\abs{\Delta} < \abs{H}^3$, the isolated root
is at distance $\abs{H}^{1/2}$ from each of the other two roots, which are at
distance $\abs{\Delta}^{1/2}/\abs{H}$ from each other, and it is the unique
root $x$ of $f$ with $\abs{f'(x)} = \abs{H}$.
\item \label{dist.none} Otherwise, the three pairwise distances all equal
$\abs{\Delta}^{1/6}$.
\end{enumerate}
\end{thm}

\begin{proof}
The differences $x_1 - x_2$, $x_2 - x_3$, $x_3 - x_1$ sum to zero.  Expanding
shows that their second elementary symmetric function is $3(x_1 x_2 + x_1 x_3
+ x_2 x_3) - (x_1 + x_2 + x_3)^2 = -H$, and their product is $\pm\delta$, of
length $\abs{\Delta}^{1/2}$.  Lemma~\ref{length.dichotomy}, applied with $P =
-H$ and $\abs{Q} = \abs{\Delta}^{1/2}$, shows that $f$ has an isolated root if
and only if $\abs{\Delta} < \abs{H}^3$, and gives the distances in
\ref{dist.isolated} and \ref{dist.none}.  In case \ref{dist.isolated},
$\abs{f'(x)}$ is the product of the distances from a root $x$ to the other two
roots, which is $\abs{H}$ for the isolated root and $\abs{H}^{1/2} \cdot
\abs{\Delta}^{1/2}/\abs{H} < \abs{H}$ for each of the others.

Next, the discriminant root is defined exactly when $\abs{\Delta} <
\abs{H}^3$.  Both conditions fail if $H = 0$.  Otherwise
Proposition~\ref{B.conv}, applied to $\sum_n \binom{3n}{n} z^n$ if $\car C \ne
3$ and to $\B_{3,-1}$ if $\car C = 3$, shows that the series converges at
$\Delta/H^3$ if and only if $\abs{\Delta/H^3} < 1$.

It remains to show that if $\abs{\Delta} < \abs{H}^3$, then the discriminant
root $\rt$ satisfies $\abs{f'(\rt)} = \abs{H}$, so that it is the isolated root
by \ref{dist.isolated}.  Put $d := \Delta/H^3$, so $\abs{d} < 1$.

Suppose first that $\car C \ne 3$.  Translating $t$ by $c_1/3$ changes neither
$\Delta$, $H$, nor the values of $f'$ at corresponding roots, and carries
\eqref{n3.nondepressed} to the discriminant root of Theorem~\ref{n3.soln}, so
we may assume that $f = t^3 + pt + q$, with $H = -3p$ and $d = -\Delta/(27p^3)$.
As $\binom{3n}{n} = 3\binom{3n-1}{n-1}$ for $n \ge 1$, the discriminant series
converges to $\limit$ with $|\limit - 1| \le |3|\,|d| < |3|$; in particular
$|\limit| = 1$.  Since $27q^2 = -(4 - 27d)p^3$ and $(4 - 27d)\limit^3 = 1 +
3\limit$ (see the proof of Theorem~\ref{n3.soln}), the discriminant root $\rt =
3q\limit/p$ satisfies
\[
  f'(\rt) = 3\rt^2 + p = p\bigl(1 - (4 - 27d)\limit^2\bigr)
  = -\frac{p\,(2\limit + 1)}{\limit}.
\]
As $|2(\limit - 1)| < |3|$, we have $|2\limit + 1| = |3 + 2(\limit - 1)| = |3|$,
so $|f'(\rt)| = |3p| = |H|$.

Now suppose that $\car C = 3$.  The series $\B_{3,-1}(d)$ has constant term $1$,
so it converges to some $\limit$ with $\abs{\limit - 1} < 1$; in particular
$\abs{\limit} = 1$.  Here $\rt = c_2/c_1 - c_1\limit$ and $f'(t) = 2c_1 t + c_2
= c_2 - c_1 t$, so $f'(\rt) = c_1^2\limit$ and $\abs{f'(\rt)} = \abs{c_1}^2 =
\abs{H}$.
\end{proof}

\subsection*{Connection with Serre's question} For the generic case from
Definition~\ref{generic.case} with $n = 3$, the discriminant root is the
isolated root.  Indeed, $\pi$ does not divide $H = c_1^2 - 3c_2$, since $\pi$
has positive degree in $c_3$ and the nonzero polynomial $H$ does not.  Thus $H$
is a unit of $A$, and $\abs{\Delta} < 1 = \abs{H}^3$, so
Theorem~\ref{dist.root.thm} applies.

When $\car C \ne 2$, the same results exhibit explicitly that the quadratic
factor $r$ from Proposition~\ref{serre.prop} is irreducible and ramified.  The
roots of $r = f/(t - \rt)$ are the two roots $z, z'$ of $f$ other than the
discriminant root $\rt$, and Theorem~\ref{dist.root.thm}\ref{dist.isolated} gives $\abs{z - z'}
= \abs{\Delta}^{1/2}/\abs{H} = \abs{\Delta}^{1/2}$.  As $\car C \ne 2$, we have
$\abs{\Delta} = \abs{\pi} = 1/2$, so $\abs{z - z'} = 2^{-1/2}$, which is not in the
value group $2^{\Z}$ of $C$, so $z - z' \notin C$.  As $z + z' \in C$, it
follows that $z, z' \notin C$, so $\rt$ is the only root of $f$ in $C$,
recovering \eqref{serre.cubic}.  It also follows that $C(z) \cong C[t]/(r)$ is a
quadratic extension of $C$ containing $z - z'$, so its value group contains
$2^{-1/2}$.  Hence $C(z)$ has ramification index $2$ over $C$; that is, $r$ is
ramified.

\section{Shortest roots over complete non-archimedean fields} \label{short.sec}

We keep the assumption that $C$ is complete with a non-archimedean absolute
value.  Over $\R$, the companion paper \cite{LongRoot} shows that when $p > 0$
or $\Delta > 0$, the trinomial root of Theorem~\ref{trinomial.thm}, continued
analytically where the series diverges, is the shortest root of a depressed
cubic.  We prove the analogous statement here and compare shortest
with isolated roots (Proposition~\ref{short.root.prop}).  This explains why the
trinomial series usually fails to answer Serre's question, and why residue
characteristic $2$ is exceptional.

The roots of $f(t) = t^3 + pt + q \in C[t]$ sum to zero, so by
Lemma~\ref{length.dichotomy} either all three have the same length or one root
is strictly shorter than the other two; we call it the \emph{shortest} root.

\begin{prop}\label{short.root.prop}
Let $f(t) = t^3 + pt + q \in C[t]$.
\begin{enumerate}[(a)]
\item \label{short.yes} The trinomial root of $f$ is defined if and only if
$f$ has a shortest root, and it is then the shortest root.
\item \label{short.2unit} If $\abs{2} = 1$, then $f$ does not have both a
shortest root and an isolated root.
\item \label{short.2small} If $\abs{2} < 1$, then $f$ has a shortest root if
and only if it has an isolated root, and then the trinomial root is the
discriminant root.
\end{enumerate}
\end{prop}

\begin{proof}
We first translate both kinds of root into inequalities.
Lemma~\ref{length.dichotomy}, applied with $P = p$ and $Q = q$, shows that $f$
has a shortest root exactly when $\abs{q}^2 < \abs{p}^3$, and that this root
then has length $\abs{q}/\abs{p}$ while the other two have length
$\abs{p}^{1/2}$.  Likewise, since $H = -3p$, Theorem~\ref{dist.root.thm} shows
that $f$ has an isolated root exactly when $\abs{\Delta} < \abs{3p}^3$, which
never happens in characteristic $3$.

\ref{short.yes}: If $p = 0$, there is neither a trinomial root nor a shortest
root.  If $p \ne 0$, the trinomial series is $\B_3$ evaluated at $q^2/(-p)^3$,
so by Proposition~\ref{B.conv} it converges exactly when $\abs{q}^2 <
\abs{p}^3$, which is the condition for a shortest root.  In that case the series
has constant term $1$ and later terms of length $< 1$, so its sum $\limitt$ has
$\abs{\limitt} = 1$.  Hence the trinomial root $\rtt = -q\limitt/p$ has length
$\abs{q}/\abs{p}$, which singles it out as the shortest root.

Parts \ref{short.2unit} and \ref{short.2small} compare the two inequalities.
As $\Delta = -4p^3 - 27q^2$, this comes down to comparing $\abs{4p^3}$ with
$\abs{27q^2}$, which depends on $\abs{2}$ and $\abs{3}$; recall that
$\abs{\ell} < 1$ for at most one prime $\ell$.

\ref{short.2unit}: If $\abs{2} = 1$ and $\abs{q}^2 < \abs{p}^3$, then
$\abs{4p^3} = \abs{p}^3$ and $\abs{27q^2} \le \abs{q}^2 < \abs{p}^3$, so
$\abs{\Delta} = \abs{p}^3 \ge \abs{3p}^3$, and $f$ has no isolated root.

\ref{short.2small}: If $\abs{2} < 1$, then $\abs{3} = 1$, so $f$ has an isolated
root if and only if $\abs{\Delta} < \abs{p}^3$; and $\abs{4p^3} < \abs{p}^3$
for $p \ne 0$.  By the ultrametric inequality, $\abs{\Delta} < \abs{p}^3$ if
and only if $\abs{27q^2} = \abs{q}^2 < \abs{p}^3$, so $f$ has an isolated root
if and only if it has a shortest root.  In that case the shortest root $x_1$ is
the trinomial root by \ref{short.yes}, and the other two roots have length
$\abs{p}^{1/2} > \abs{x_1}$.  So the two sides of the triangle through $x_1$
have length $\abs{p}^{1/2}$, the maximum possible, and the short side joins
$x_2$ and $x_3$.  Thus $x_1$ is the isolated root, which is the discriminant
root by Theorem~\ref{dist.root.thm}.
\end{proof}

\subsection*{Connection with Serre's question}
Suppose $\car k \ne 3$, and consider the depressed generic cubic $g$ of
\eqref{depressed}.  It has an isolated root, because $\abs{\Delta} < 1 =
\abs{H}^3$ (see the end of \autoref{nonarch.sec}).  If $\car k \ne 2$, then
$\abs{2} = 1$, so $g$ has no shortest root by
Proposition~\ref{short.root.prop}\ref{short.2unit}: the trinomial series
diverges and does not answer Serre's question.  If $\car k = 2$, then
$\abs{2} = 0$, so by Proposition~\ref{short.root.prop}\ref{short.2small} the
trinomial series converges and the trinomial root is the discriminant root.  In
fact the two series coincide here, since $\Delta = q^2$ and $\binom{3n}{n} =
(2n+1) A_n(3,1) \equiv A_n(3,1) \bmod 2$.

\section{Quartic polynomials}  \label{quartic.sec}

We return to the generic case with $n=4$ and compute the ramified quadratic factor $r(t)$ of
Proposition~\ref{serre.prop}, together with the unramified factor
$u(t)$; see Theorem~\ref{quartic.prop}.  We assume $\car k \ne 2, 3$
throughout, to control the length of the paper.  Then $\pi = \Delta$ is a
uniformizer of $A$.

Substituting $t \mapsto t - c_1/4$ into the generic quartic $f(t)$ gives the
depressed quartic
\begin{equation} \label{depressed.quartic}
g(t) = t^4 + ct^2 + dt + e \quad \in K[t]
\end{equation}
for some $c, d, e \in k[c_1, \ldots, c_4]$.  The substitution is invertible, so
$k[c_1, \ldots, c_4] = k[c_1, c, d, e]$ is a polynomial ring in $c_1, c, d, e$.
It carries the factors $r(t)$ and $u(t)$ of $f$ from
Proposition~\ref{serre.prop} to monic quadratic factors of $g$ in $A[t]$, the
first ramified and the second unramified; we abuse notation and denote these
again by $r(t)$ and $u(t)$.  Our aim is to compute them.

We will repeatedly need to know that certain polynomials in $c, d, e$ are units
of $A$.  A convenient bookkeeping device is the weighted grading on $k[c, d,
e]$ in which $c$, $d$, $e$ have weights $2$, $3$, $4$.  In this grading
$\Delta$ is homogeneous of weight $12$, as one checks term by term from
\[
\Delta = 16c^4 e - 4c^3 d^2 - 128 c^2 e^2 + 144 c d^2 e - 27 d^4 + 256 e^3.
\]
Consequently a nonzero homogeneous polynomial of weight less than $12$ is not
divisible by $\Delta$, so it is a unit of $A$.  For instance, $c^2 + 12e$ has
weight $4$, and the numerators $2c^3 - 72ce + 27d^2$, $c^3 - 20ce + 9d^2$, and
$8ce - 2c^3 - 9d^2$ appearing below have weight $6$, so all of these are units
of $A$.

Write $r_1, r_2$ for the roots of $r(t)$ in some algebraic closure of $C$ and
$u_1, u_2$ for the roots of $u(t)$.  Since the two polynomials are irreducible
and $r(t) \ne u(t)$, the four roots $r_1, r_2, u_1, u_2$ are distinct.  To
compute the factors, we exploit the cubic resolvent
\[
h(t) := t^3 - ct^2 - 4et + (4ce - d^2)
\]
of $g$, as defined in \cite[p.~52]{Kap:F}.  It has the same discriminant as
$g(t)$, and its roots are $r_1 r_2 + u_1 u_2$, $r_1 u_1 + r_2 u_2$, and $r_1 u_2
+ r_2 u_1$.

\begin{qprop}[$\car k \ne 2, 3$] \label{quartic.prop}
The discriminant root $\rt$ of the resolvent cubic $h$, as in
\eqref{n3.nondepressed}, is defined and lies in $A$.  There is a unique $\rho
\in A$ such that
\[
\rho^2 = \rt^2 - 4e \quad\text{and}\quad \rho \equiv \rho_0 := \frac{8ce - 2c^3
- 9d^2}{2c^2 + 24e} \pmod{\Delta},
\]
and it is a unit of $A$.  For $s := d/\rho$, we have $s^2 = \rt - c$, and the
ramified and unramified quadratic factors of the depressed generic quartic
$g(t)$ are
\[
r(t) = t^2 - s t + \tfrac12 \left( \rt + \rho \right)
\quad \text{and} \quad
u(t) = t^2 + s t + \tfrac12 \left( \rt - \rho \right).
\]
\end{qprop}

The proof will show that $\rt = r_1 r_2 + u_1 u_2$ and $\rho = r_1 r_2 - u_1 u_2$.

\begin{proof}
The substitution $t \mapsto t + c/3$ transforms $h$ into the depressed cubic
$t^3 + p_h t + q_h$ with
\[
  p_h = -\frac{c^2 + 12e}{3} \quad\text{and}\quad q_h = -\frac{2c^3 - 72ce + 27d^2}{27},
\]
which are units of $A$ by the weight count above.  Hence $\abs{\Delta} < 1 =
\abs{3p_h}^3$, so the discriminant series converges to some $\limit \in A$, and
\eqref{n3.nondepressed} provides the root $\rt = c/3 + (3q_h/p_h)\limit \in A$
of $h$.  Every term of the series after the constant term $1$ is divisible by
$\Delta$, so $\limit \equiv 1 \pmod{\Delta}$ and
\begin{equation} \label{quartic.alpha}
\rt \equiv \frac{c}{3} + \frac{3q_h}{p_h} \equiv \frac{c^3 - 20ce + 9d^2}{c^2 + 12e} \pmod{\Delta}.
\end{equation}
The element $\rt$ is the only root of $h$ in $C$, for the discriminant of $h$
equals the discriminant of $g$, which is not a square in $C$ (as in
\autoref{why.n3.sec}).

Since $g$ has no $t^3$ term, we may write $r(t) = t^2 - S t + R$ and $u(t) =
t^2 + S t + U$ with $R, S, U \in A$.  The root $r_1 r_2 + u_1 u_2 = R + U$ of
$h$ lies in $A$, so it is $\rt$, and comparison of coefficients gives
\[
\rt = R + U, \qquad c = \rt - S^2, \qquad d = S(R-U), \qquad e = RU.
\]
For $\rho := R-U$, it follows that
\[
\rho^2 = \rt^2 - 4e, \qquad
\operatorname{disc}(r) = S^2 - 4R = -c - \rt - 2\rho.
\]
Since $r$ is ramified and $\car k \ne 2$, its discriminant has odd valuation
\cite[24.4]{SeciV}, hence lies in $(\Delta)$.  Thus
$\rho \equiv -(c+\rt)/2 \equiv \rho_0 \pmod{\Delta}$ by
\eqref{quartic.alpha}.  The weight count above shows $\rho_0$, and hence
$\rho$, is a unit.  As $2\rho_0$ is a unit, Hensel's Lemma gives uniqueness of
the root of $T^2 - (\rt^2 - 4e)$ congruent to $\rho_0$.

Finally, $s = d/\rho = S$ gives $s^2 = \rt - c$, and solving
$R+U = \rt$, $R-U = \rho$ gives the stated factors.
\end{proof}

This factorization is of the sort provided by Descartes' method in
\cite{Brookfield} or \cite[Appendix, \S15]{Chrystal}, which writes $g(t) = (t^2
+ st + m)(t^2 - st + n)$ with $s^2$ a root of Descartes' resolvent $h(t + c) =
t^3 + 2ct^2 + (c^2 - 4e)t - d^2$; here $s^2 = \rt - c$.  
We have provided formulas for all of the terms ---
$\rt$ by the discriminant series of Theorem~\ref{n3.soln}, and $\rho$ by
Newton's method $\rho_{i+1} := \tfrac12 \bigl(\rho_i + (\rt^2 - 4e)/\rho_i\bigr)$
starting from $\rho_0$ --- and we have determined the \emph{signs} of $s$ and
$\rho$ via the congruence $\rho \equiv \rho_0$, thereby identifying which of the
quadratic factors is the ramified one.

\subsection*{Acknowledgements} We thank J-P.~Serre for posing his question, which stimulated this paper.  ChatGPT and Claude Opus were used for referee-style feedback on earlier versions, to refactor the text, and to explore possible new results.  All arguments and calculations have been independently checked by the authors.

\bibliographystyle{amsalpha}
\bibliography{skip_master,cubic}

\end{document}